\documentclass[11pt]{amsart}
\usepackage{amssymb,latexsym, amscd, a4wide}
\usepackage[all]{xy}
\usepackage{graphicx}
\usepackage{mathrsfs}
\usepackage{amsmath}
\usepackage{pb-diagram}
\usepackage{enumerate}

 \newtheorem{theorem}{Theorem}[section]
\newtheorem{Thm}{Theorem}[section]
 \newtheorem{cor}[theorem]{Corollary}

 \newtheorem{lemma}[theorem]{Lemma}
\newtheorem{lemma-definition}[theorem]{Lemma-Definition}
\newtheorem{proposition-definition}[theorem]{Proposition-Definition}
\newtheorem{Lem}[theorem]{Lemma}

 \newtheorem{proposition}[theorem]{Proposition} \theoremstyle{definition}
 \newtheorem{definition}[theorem]{Definition} \theoremstyle{definition}
\newtheorem{defi}[theorem]{Definition} \theoremstyle{definition}
 \theoremstyle{definition}
 \newtheorem{example}[theorem]{Example}
 \theoremstyle{remark}
 \newtheorem{rem}[theorem]{Remark}
  \newtheorem{rems}[theorem]{Remarks}
\newtheorem{rem*}[theorem]{Remark}
 \numberwithin{equation}{section}

\newcommand{\tU}{{\widetilde{U}_q}}
\newcommand{\tW}{{\tilde{\W}}}

\newcommand{\res}{\operatorname{res}}
\newcommand{\Ker}{\operatorname{Ker}}

\newcommand{\Hom}{\operatorname{Hom}}

\newcommand{\Jm}{\operatorname{Im}}

\newcommand{\rank}{\operatorname{rank}}

\newcommand{\uobar}{\overline{\uo}}

\newcommand{\MOD}{\ensuremath{{\operatorname{Mod}}}}

\newcommand{\IndBP}{\ensuremath{Ind^{P_q}_{B_q}}}

\newcommand{\IndPQ}{\ensuremath{Ind^{Q_q}_{P_q}}}

\newcommand{\IndBQ}{\ensuremath{Ind^{Q_q}_{B_q}}}
\newcommand{\IndBL}{\ensuremath{Ind^{L_q}_{(B \cap L)_q}}}

\newcommand{\ResQP}{\ensuremath{Res^{P_q}_{Q_q}}}

\newcommand{\Z}{\ensuremath{\mathcal{Z}}}

\newcommand{\ZHC}{\ensuremath{\mathcal{Z}^{HC}}}

\newcommand{\isoto}{\ensuremath{\overset{\sim}{\longrightarrow}}}

\newcommand{\C}{\ensuremath{{\mathbb{C}}}}
\newcommand{\D}{\ensuremath{\mathcal{D}}}
\newcommand{\Dq}{\ensuremath{\mathcal{D}_q}}

\newcommand{\B}{\ensuremath{\mathcal{B}}}
\newcommand{\Q}{\ensuremath{\mathcal{Q}}}
\newcommand{\Po}{\ensuremath{\mathcal{P}}}

\newcommand{\BGGcat}{\ensuremath{{\operatorname{O}}}}

\newcommand{\lf}{\ensuremath{{\lo\text{-}\hbox{\tiny{int}}}}}
\newcommand{\smallleft}{\ensuremath{{\hbox{\tiny{left}}}}}
\newcommand{\fin}{\ensuremath{{\hbox{\tiny{int}}}}}
\newcommand{\g}{\ensuremath{\mathfrak{g}}}
\newcommand{\bo}{\ensuremath{\mathfrak{b}}}

\newcommand{\n}{\ensuremath{\mathfrak{n}}}

\newcommand{\uo}{\ensuremath{\mathfrak{r}}}
\newcommand{\lo}{\ensuremath{\mathfrak{l}}}

\newcommand{\p}{\ensuremath{\mathfrak{p}}}
\newcommand{\qo}{\ensuremath{\mathfrak{q}}}
\newcommand{\h}{\ensuremath{\mathfrak{h}}}

\newcommand{\bb}{\ensuremath{\mathfrak{b}}}
\newcommand{\Loc}{\ensuremath{\mathcal{L}}}
\newcommand{\BGG}{\ensuremath{\mathcal{O}}}
\newcommand{\wMPq}{\ensuremath{\wt{M}_{P_q}}}

\newcommand{\MPql}{\ensuremath{{M}_{P_q,\lambda}}}

\newcommand{\wMBq}{\ensuremath{\wt{M}_{B_q}}}
\newcommand{\wMQq}{\ensuremath{\wt{M}_{Q_q}}}

\newcommand{\MBql}{\ensuremath{{M}_{B_q,\lambda}}}

\newcommand{\pQPlb}{\ensuremath{\pi^{\mathcal{Q}q}_{\mathcal{P}q\bullet}}}
\newcommand{\pQPub}{\ensuremath{\pi^{\mathcal{Q}q\bullet}_{\mathcal{P}q}}}
\newcommand{\pPBlb}{\ensuremath{\pi^{\mathcal{P}q}_{\mathcal{B}q\bullet}}}
\newcommand{\pPBub}{\ensuremath{\pi^{\mathcal{P}q\bullet}_{\mathcal{B}q}}}

\newcommand{\pGPls}{\ensuremath{\pi^{\mathcal{G}q}_{\mathcal{P}q*}}}

\newcommand{\pQPls}{\ensuremath{\pi^{\mathcal{Q}q}_{\mathcal{P}q*}}}
\newcommand{\pQPus}{\ensuremath{\pi^{\mathcal{Q}q*}_{\mathcal{P}q}}}

\newcommand{\UA}{\ensuremath{\U_\mathcal{A}}}

\newcommand{\UAres}{\ensuremath{\U^{\res}_\mathcal{A}}}

\newcommand{\A}{\ensuremath{\mathcal{A}}}
\newcommand{\OA}{\ensuremath{\mathcal{O}_\mathcal{A}}}

\newcommand{\W}{\ensuremath{\mathcal{W}}}

\newcommand{\U}{\ensuremath{\operatorname{U}}}
\newcommand{\Uq}{\ensuremath{\operatorname{U}_q}}
\newcommand{\Uql}{\ensuremath{\operatorname{U}^{\lambda}_q}}
\newcommand{\Ug}{\ensuremath{\operatorname{U(\g)}}}

\newcommand{\DtP}{\ensuremath{\widetilde{\mathcal{D}}_{\mathcal{P}_q}}}

\newcommand{\UqP}{\ensuremath{\operatorname{U}^{\lo\text{-}{\fin}}_q}}
\newcommand{\UqQ}{\ensuremath{\operatorname{U}^{\lo'\text{-}{\fin}}_q}}
\newcommand{\DqP}{\ensuremath{\Dq^{\lo\text{-}{\fin}}}}
\newcommand{\DqQ}{\ensuremath{\Dq^{\lo'\text{-}{\fin}}}}
\newcommand{\Uqf}{\ensuremath{\operatorname{U}^{{\fin}}_q}}
\newcommand{\Dqf}{\ensuremath{\Dq^{\text{}{\fin}}}}

\newcommand{\Uqres}{\ensuremath{\U^{\res}_q}}

\newcommand{\DtQ}{\ensuremath{\widetilde{\mathcal{D}}_{\mathcal{Q}_q}}}
\newcommand{\DtB}{\ensuremath{\widetilde{\mathcal{D}}_{\mathcal{B}_q}}}

\newcommand{\DPl}{\ensuremath{{\mathcal{D}}^{\lambda}_{\mathcal{P}_q}}}

\newcommand{\MPG}{\ensuremath{\MOD(\mathcal{O}_q, P_q)}}

\newcommand{\DPG}{\ensuremath{\MOD(\DqP,P_q,\Uq(\uo))}}

\newcommand{\DQG}{\ensuremath{\MOD(\DqQ,Q_q,\Uq(\uo'))}}
\newcommand{\DPGl}{\ensuremath{\MOD_\lambda(\DqP,P_q,\Uq(\uo))}}
\newcommand{\DPGwl}{\ensuremath{\MOD_{\widehat{\lambda}}(\DqP,P_q,\Uq(\uo))}}

\newcommand{\OPG}{\ensuremath{\MOD(\Oq,P_q)}}

\newcommand{\OQG}{\ensuremath{\MOD(\Oq,Q_q)}}
\newcommand{\OGG}{\ensuremath{\MOD(\Oq,G_q)}}

\newcommand{\wt}{\ensuremath{\widetilde}}

\newcommand{\alp}{\ensuremath{{\alpha}_{\lo}}}
\newcommand{\wtalp}{\ensuremath{\widetilde{\alpha}_{\lo}}}

\newcommand{\ot}{\ensuremath{\otimes}}

\newcommand{\Oq}{\ensuremath{\mathcal{O}_q}}

\begin{document}
\title{Singular localization for Quantum groups at generic $q$.}

\address{Erik Backelin, Departamento de Matem\'{a}ticas, Universidad de los Andes,
Carrera 1 N. 18A - 10, Bogot\'a, Colombia}
\email{erbackel@uniandes.edu.co}
\address{Kobi Kremnitzer, Mathematical Institute, University of Oxford, 24ï¿½29 St Giles'
Oxford OX1 3LB, UK}
 \email{kremnitzer@maths.ox.ac.uk}
\date{}
\author{Erik Backelin and Kobi Kremnizer \\ {\bf Accepted for publication in Advances in Mathematics}}
\begin{abstract} We quantize parabolic flag manifolds and describe categories of quantum $\D$-modules on
them at a singular central character. We compute global sections
of generators for these categories for any $q \in \C^*$. For generic  $q$ we prove a singular version of
Beilinson-Bernstein localization for a quantized enveloping
algebra.
\end{abstract}
\maketitle

\section{Introduction} \subsection{} This note is part of our ongoing
project on localization and representation theory of quantum
groups. Localization theory starts with the celebrated
localization theorem of Beilinson and Bernstein, \cite{BB81},
which we remind goes as follows: Let $\g$ be a complex semi-simple
Lie algebra, $\h$ a Cartan subalgebra and $\B$ the flag manifold
of $\g$. Let $\lambda \in \h^*$ be regular and dominant and let
$I_\lambda$ be the corresponding maximal ideal in the center of
$\Ug$. Let $\D^\lambda_{\B}$ be the sheaf of $\lambda$-twisted
differential operators on $\B$. Then $\Gamma(\D^\lambda_{\B})
\cong \Ug^\lambda := \Ug/(I_\lambda)$ and $\Gamma:
\MOD(\D^\lambda_{\B}) \to \MOD(\Ug^\lambda)$ is an equivalence of
categories. For applications and details we recommend the book
\cite{HTT08}.

The next fundamental step was taken by Bezrukavnikov, Mirkovi\'c and
Rumynin, \cite{BMR}. They did Beilinson-Bernstein localization in
finite characteristic at regular central character and later in
\cite{BMR2} at singular central character, at the level of derived
categories, utilizing the techniques of Azumaya algebras.

The authors did localization for a quantum group $\Uq := \Uq(\g)$
at a generic $q \in \C^*$ in \cite{BK06} and at a root of unity in
\cite{BK08} - in both papers for regular central character, in the
latter motivated by the ideas of \cite{BMR}. In \cite{BK10} we
also did localization for the complex enveloping algebra case at a
singular central character.

\subsection{}
In this paper we do singular localization for $\Uq$ at a generic $q$. Let us
sketch the basic ideas:

We shall merely assume that $\g$ is reductive and
let $G$ be a reductive group such that $Lie \, G = \g$. Let $P
\subseteq G$ be a parabolic subgroup and let $\Po = G/P$ be the
corresponding parabolic flag manifold. First we quantize $\Po$ the
same way as we quantized $\B$ in \cite{BK06}. We remind that this
is done as follows: Observe that the category
$\MOD(\BGG_{\Po})$ of quasi-coherent sheaves on $\Po$ is
equivalent to the category $\MOD(\BGG(G),P)$ of
$P$-equivariant $\BGG(G)$-modules, since $G$ is affine. Since an
algebraic $P$-action is the same thing as an $\BGG(P)$-coaction
the latter category admits a quantization.

Indeed, let $\Oq := \Oq(G)$ and $\Oq(P)$ be the quantized Hopf
algebras of functions on $G$ and $P$, respectively, and let
$\MOD(\Oq,P_q)$ be the category whose objects are
$\Oq$-modules and $\Oq(P)$-comodules with a certain equivariance
compatibility, see Section \ref{Definition of the quantum flag
manifold}. According to Grothendieck a space is the same thing as
its category of sheaves, so we think of $\MOD(\Oq,P_q)\mod$ as
a quantization ``$\Po_q$" of $\Po$.  In Proposition \ref{aff} we give
a quantum counterpart to Serre's description of projective
varieties.

Then we fix a weight $\lambda$ and chose $P$ such that the
singular roots of $\lambda$ are contained in the $P$-parabolic
roots. \cite{BMR2} considered a sheaf $\mathcal{D}^{\lambda}_\Po$
of certain extended differentials operators in characteristic $p$
on $\Po$ that locally looks like $\D_{\Po}$ tensored with the
primitive quotient determined by $\lambda$ of the enveloping
algebra of the Levi-factor $\lo$ of $\p := Lie \, P$. In \cite{BK10} we
considered the same sheaf in characteristic $0$.

We quantize the category
$\MOD(\mathcal{D}^{\lambda}_\Po\text)$ in Section
\ref{Definition of quantum D-modules}. As in our earlier work we have to work with ad-finite (here called ad-integrable) versions of $\Uq$; a novelty is that we use a different ad-finite version $\UqP \subseteq \Uq$
for each parabolic.
  In Theorem \ref{thmtag0} we compute global
sections of the distinguished object  $\DqP$  that represents global sections (morally, this is the sheaf of quantum differential operators on $\Po$).  This result is valid for all $q$ (except perhaps
roots of unity of order smaller than the Coxeter number of $\g$).  The proof reduces to $\Po = \B$ in which case the result was established in \cite{BK06, BK08}; a short self-contained proof was given in \cite{BK11}.

Another result of independent interest, valid for all $q$, is the parabolic triangular decomposition, Proposition-Definition \ref{Ur prop def}, that we deduce from the Majid-Radford theorem. (For generic $q$ this was done in \cite{G11}.)

\smallskip

Our main result, which holds for generic $q$, is a version of Beilinson-Bernstein
localization, Theorem \ref{singlocthmtag0}. The proof is close to
that given in \cite{BK10}, which in turn is a variation of the
argument of \cite{BB81}.

T. Tanisaki \cite{T05,T12} has a different approach to quantum localization in the regular case. His approach follows on from the work of Lunts and Rosenberg \cite{LR97,LR99}. In this approach quantum flag varieties are defined as non-commutative projective varieties using graded algebras. Differential operators are defined using graded algebras as well.  In \cite{BK06} we show that the graded algebra approach (on the level of quasi-coherent sheaves) is equivalent to our equivariant approach. This result should also extend to D-modules. Hence the two approaches are essentially equivalent. The equivariant approach is more geometric since it views quantum spaces as quotient spaces. This allows for easy reductions of the problems into
questions in the representation theory of quantum groups.  Tanisaki proves a Beilinson-Bernstein localization theorem for quantum groups at regular central character for generic $q$ equivalent to our result in \cite{BK06}.

\subsection{}
In a forthcoming paper we will do singular localization
at a root of unity. Those results here which are valued for all $q$ will be needed in that paper. Just as in the modular case
(regular and singular) Beilinson-Bernstein
localization will then only hold at the level of derived
categories. This is the most interesting case and we shall use
this and the results of \cite{BMR, BMR2, BM} to compare the
representation theory of $\Uq$ with the representation theory of
the Lie algebra $\g(\overline{\mathbb{F}}_p)$, when $q^p=1$.

\medskip

\noindent
Concerning other applications, let us mention that  we in \cite{BK10}  used singular localization to describe translation functors, singular blocks in the
Bernstein-Gelfand-Gelfand category $\BGGcat$, Harish Chandra (bi-)modules and Whittaker modules in the enveloping algebra case. \textit{Exactly the same can be done for the generic quantum group, with practically exactly the same arguments.} We have omitted
to write this down.

\subsection{}
We advise the reader to read \cite{BK10} before this article.
That paper was written with this one in mind and the geometric
ideas behind the equivariant definitions given here are there thoroughly explained.
Quantum groups are technically harder to work with than
enveloping algebras in the context of localization theory because,
for instance, the adjoint action of $\Uq$ on itself is not
integrable (see Section \ref{loc fin section}) and PBW-bases and
(parabolic) triangular decompositions are more complicated.

However, once those technical complications are overcome the conceptual distance between the generic quantum group
and the complex enveloping algebra case is small.

\subsection{Acknowledgements} Part of this work was done during the first authors visit to the university of Aarhuus
in the spring of 2011 where excellent working conditions and great hospitality were provided. We would like to thank Henning H. Andersen for many useful conversation.

\section{Preliminaries on quantum groups} We work over $\C$, $\ot = \ot_\C$. $t$ is a parameter and
$q$ denotes a complex invertible number such that $q^2
\neq 1$. We say that $q$ is \emph{generic} if $q$ is not a root of
unity.

$\MOD(S)$ denotes the category of left modules over a ring $S$, $\MOD(S, \, additional \; data)$ is the category whose objects are left $S$-modules with some $additional \; data$.

In this section we recall some facts about quantum groups. The
material here is mostly standard. This paper is a continuation of
the papers \cite{BK06, BK08}. Let us mention that we shall not
particularly follow the notations of those papers, but rather
``quantize" those of \cite{BK10}.

\cite{CP} is our main reference for the material here. See also
Section \ref{appendix} for some facts about Hopf algebras that
will be used here.

\subsection{Root data}\label{Root data} Let $\g$ be a reductive Lie
algebra and let $\h \subseteq \bb \subseteq \g$ be a Cartan
subalgebra contained in a Borel subalgebra. Let $\n \subset \bb$
be the unipotent radical. Let $\overline{\bb}$ be the opposite
Borel and $\overline{\n}$ its unipotent radical. We denote by
$\Ug$ the enveloping algebra of $\g$ and by $\Z(\g)$ the center of
$\U(\g)$.

Let $\Delta$ be the simple roots, let $\Lambda$ be the lattice of
integral weights and let $\Lambda_r$ be the root lattice. Let
$\Lambda_+$ and $\Lambda_{r +}$ be the positive weights and the
positive integral linear combinations of the simple roots,
respectively.

Let $\mathcal{W}$ be the Weyl group of $\g$. We let $\langle
\,,\,\rangle$ denote a $\mathcal{W}$-invariant bilinear form on
$\h^\star$ normalized by $\langle \gamma,\gamma \rangle = 2$ for
each short root $\gamma$.

Let $T_\Lambda := \Hom_{groups}(\Lambda,\C^\star) = Maxspec \, \C
\Lambda$ be the character group of $\Lambda$, where $\C \Lambda$
is the group algebra of $\Lambda$. The $\W$-action on $\Lambda$
induces a $\W$-action on $T_\Lambda$. We define the
$\bullet$-action of $\W$ on $T_\Lambda$ by $w \bullet \lambda =
w(\lambda + \rho)-\rho$, where $\rho$ is the half sum of the
positive roots. For $\mu \in \Lambda$ we define $q^\mu \in
T_\Lambda$ by the formula $q^\mu(\gamma) =
q^{\langle\mu,\gamma\rangle}$, for $\gamma \in \Lambda$. For any
$\alpha \in \Delta$, put $d_\alpha := \langle \alpha , \alpha
\rangle/2$.

Let $G$ be a connected reductive algebraic group such that
$G/\Z(G)$ is simply connected (where $\Z(G)$ is the center of $G$)
and $Lie \, G = \g$. Let $B$ be the Borel subgroup of $G$ with
$Lie \, B = \bb$. Let $P \supseteq B$ be a parabolic subgroup of
$G$ and let $\p = Lie \, P$. Let $R$ be the unipotent radical and
let $L$ be the Levi-factor of $P$ and denote by $\uo$ and $\lo$
their respective Lie algebras. Let $\Delta_P \subseteq \Delta$ be
the $P$-parabolic roots (so $\Delta_B = \emptyset$, $\Delta_G =
\Delta$). Write $\overline{P}$, $\overline{R}$, $\overline{\p}$
and $\overline{\uo}$ for their respective opposite groups and Lie
algebras.

Let $\lambda \in T_\Lambda$ and put $\Delta_\lambda = \{\alpha \in
\Delta; s_\alpha \bullet \lambda = \lambda\}$. We say that
\begin{itemize}
\item $\lambda$ is $P$-{\it regular} if $\Delta_\lambda \subseteq
\Delta_P$. $\lambda$ is regular if it is $B$-regular, i.e. if
$Stab_{(\W, \bullet)}(\lambda) = \{e\}$. \item $\lambda$ is a
$P$-\emph{character} if $\lambda$ is integral and
$\lambda(K_\alpha) = 1$, for $\alpha \in {\Delta_P}$.
\end{itemize}

\subsection{Quantized enveloping algebras}\label{Quantized enveloping algebras}
\subsubsection{}  Let
$(a_{\alpha \beta})_{\alpha, \beta \in \Delta}$ be the Cartan
matrix of $\g$. Chose integers $d_{\alpha}$ so that $(d_\alpha
a_{\alpha \beta})$ is symmetric and define a new bilinear form
$\langle \ , \ \rangle_d$ on $\Lambda$ by $\langle \mu , \nu
\rangle_d = \sum d_\alpha f_\alpha \langle \alpha, \nu \rangle,$
for $\mu, \nu \in \Lambda$ and $\mu = \sum f_\alpha \alpha$,
$f_\alpha \in \mathbb{Q}$. Put $q_\alpha = q^{d_\alpha}$.

Let $\Uq := \Uq(\g)$ be the simply connected quantized enveloping
algebra of $\g$. Recall that $\Uq$ has $\C$-algebra generators
$E_\alpha, F_\alpha, K_\mu$, for $\alpha, \beta \in \Delta$ and
$\mu \in \Lambda$. These are subject to the relations
$$ K_{\lambda} K_{\mu} = K_{\lambda+\mu},\;\; K_0 = 1,$$
$$K_\mu E_\alpha K_{-\mu} = q^{\langle \mu, \alpha \rangle_d} E_\alpha,\;\; K_\mu F_\alpha K_{-\mu} = q^{-\langle \mu, \alpha\rangle_d}
F_\alpha,$$
$$[E_\alpha,\, F_\beta] = \delta_{\alpha, \beta} {{K_\alpha - K_{-\alpha}}\over{q_{\alpha}-q^{-1}_{\alpha}}}$$
and certain Serre-relations that we do not recall here. We have
$$
\triangle K_\mu = K_\mu \ot K_\mu, \  \triangle E_\alpha =
K_{\alpha} \ot E_\alpha + E_\alpha \ot 1, \ \triangle F_\alpha = 1
\ot F_\alpha + F_\alpha \ot K_{-\alpha}
$$
$$
S(K_\mu) = K_{-\mu}, \ S(E_\alpha) = -K_{-\alpha} E_\alpha , \
S(F_\alpha) = -F_\alpha K_{\alpha}
$$
$$
\epsilon(K_\mu) = 1, \epsilon(E_\alpha) = \epsilon(F_\alpha) = 0
$$

\subsubsection{}
Let $\Oq = \Oq(G)$ be the algebra of matrix coefficients of finite
dimensional type-1 representations of $\Uq$. This is a
quantization of the algebra of functions $\BGG(G)$ on $G$. There
is a natural pairing $\langle \;,\;\rangle: \Uq \otimes \Oq \to
\C$. This gives a $\Uq$-bimodule structure on $\Oq$ as follows
\begin{equation}\label{23}
\mu_l(u)(a) := ua :=  a_1 \langle u,a_2\rangle,\;\; \mu_r(u)(a) :=
au := \langle u,a_1\rangle a_2, \,\, u \in \Uq, a \in \Oq
\end{equation}
so that $\mu_l$ is a left action and $\mu_r$ is a right action.
Then $\Oq$ is the (restricted) dual of $\Uq$ with respect to this
pairing. Quantizing the enveloping algebras $\U(\p)$ and $\U(\lo)$
gives Hopf subalgebras of $\Uq$:
$$
\Uq(\p) = \C\langle K_\mu, E_\alpha, F_\beta; \mu \in \Lambda,
\alpha \in \Delta , \beta \in \Delta_P\rangle \hbox{ and }
$$
$$
\Uq(\lo) = \C\langle K_\mu, E_\alpha, F_\beta; \mu \in \Lambda,
\alpha, \beta \in \Delta_P\rangle
$$
In particular, $\Uq(\h)$ is isomorphic to the group algebra
$\C\Lambda$.

There is the counit $\epsilon: \Uq \to \C$. We put ${\Uq}_{>0} =
\Ker \epsilon$ and for any subalgebra $R$ of $\Uq$ we put $R_{>0}
= R \cap {\Uq}_{>0}$.

A quantization of $\Uq(\uo)$ will be given
in Section \ref{Quantizing U(uo)}.

\subsubsection{} We let $\Oq(P)$ and $\Oq(L)$ be the
quotient Hopf-algebras of $\Oq$ corresponding to the subalgebras
$\Uq(\p)$ and $\Uq(\lo)$ of $\Uq$, respectively, by means of the
duality between $\Oq$ and $\Uq$.

\subsection{Modules and comodules}\label{Modules and comodules}
\subsubsection{} We shall often call a right (resp., left) $\Oq
\text{-comodule}$ a left (resp., right) $G_q\text{-module}$. For a
(right) $G_q$-module $M$ we denote by $M^{G_q} = \{m \in M;
\triangle m = 1 \ot m\}$ the set of $G_q$-invariants. (Similarly,
there are $P_q$-,$L_q$-modules, etc.)

Let $Q \supseteq P$ be a parabolic subgroup of $G$. $\Oq(Q)$ is
naturally a $Q_q-P_q$-bimodule. Using the antipode we can make a
right $P_q$-module into a left $P_q$-module. Because of this we
shall freely pass between  $Q_q-P_q$-bimodules and vector spaces
equipped with commuting left $Q_q$ and $P_q$-module structures.

We have an
adjoint pair of functors
\begin{equation}\label{res-ind def}
Res^{P_q}_{Q_q}: \MOD(Q_q) \rightleftarrows \MOD(P_q) :Ind^{Q_q}_{P_q}
\end{equation}
where $Res^{P_q}_{Q_q}(M) = M$ as a set and the $P_q$-module
structure is the restriction of the of $Q_q$-module structure,
i.e. the $\Oq(P)$-comodule structure is the composition $M
\overset{\triangle}{\to} M \ot \Oq(Q) \to \ M \ot \Oq(P)$, for $M
\in \MOD(Q_q)$. $Ind^{Q_q}_{P_q}(N) = (\Oq(Q) \ot N)^{P_q}$, for $N
\in \MOD(P_q)$ and the $P_q$-invariants are taken with respect to
the diagonal $P_q$-action. The $Q_q$-action on
$Ind^{Q_q}_{P_q}(N)$ is given by the left $\Q_q$-action on
$\Oq(Q)$.

\subsubsection{} For $M \in \MOD(Q_q)$ and $N \in \MOD(P_q)$ there is the \emph{tensor
identity}
\begin{equation}\label{tensor identity}
M \ot Ind^{Q_q}_{P_q} N \overset{\sim}{\to} Ind^{Q_q}_{P_q} (Res^{P_q}_{Q_q}(M) \ot
N), \ m \ot (a \ot n) \mapsto m_1a \ot m_2 \ot n
\end{equation}
which is an isomorphism of $Q_q$-modules.
\medskip

\subsubsection{} Suppose that $\lambda \in T_\Lambda$. We observe that
there exists an irreducible left $P_q$-module $V_{P_q}(\lambda)$
with highest weight $\lambda$ iff $\lambda(K_\alpha) \in
\{1,q,q^2, \ldots \}$, for $\alpha \in \Delta_P$, when $q$ is
generic (at a root of unity there is a similar condition). Note
that $V_{L_q}(\lambda) := V_{P_q}(\lambda)$ is an irreducible
representation for $L_q$. Of course, $\dim V_{P_q}(\lambda) = 1
\iff \lambda$ is a $P_q$-character.

\subsubsection{} Let $P \subseteq Q$ be parabolic subgroups of $G$ and let $L'$
be the Levi factor of $Q$. We state for the record
\begin{lemma}\label{induction L verses P} For
any $P_q$-module $M$, $Ind^{L'_q}_{(L' \cap P)_q} M$ carries a
natural structure of $Q_q$-module. There is a natural isomorphism
of $Q_q$-modules
\begin{equation}\label{general equality}
\tau: \IndPQ(M) \overset{\sim}{\to} Ind^{L'_q}_{(L' \cap P)_q}(M).
\end{equation}
\end{lemma}
\begin{proof}
Let $\Oq(Q) \overset{p}{\to} \Oq(L')$ be the algebra homomorphism
which is dual to the inclusion $\Uq(Lie \, L') \hookrightarrow
\Uq(Lie \, Q)$. We then define $\tau(f \ot m) = p(f) \ot m$, for
$f \ot m \in \IndPQ(M)$. It is straightforward to verify that this
is an isomorphism. The $Q_q$-action on the right hand side is now
defined by transportation of structure.
\end{proof}

\subsection{The center of $\Uq$ and the Harish-Chandra
homomorphism}\label{The center of Uq and the Harish-Chandra
homomorphism} \subsubsection{} Let $\Z(A)$ denote the center of an
algebra $A$. Put $\Z = \Z(\Uq)$. Then $\Z$ contains the
Harish-Chandra center $\Z^{HC}$ and, if $q$ is a primitive $l$'th
root of unity, $\Z$ also contains the $l$-center $\Z^{(l)}$, which
is generated by $E^l_{\alpha}, F^l_\alpha$ and $K^l_{\mu}$,
$\alpha \in \Delta, \mu \in \Lambda$.

Let us now describe $\Z^{HC}$. Let $\Gamma$ be the group of all
group homomorphisms from $\Lambda$ to $\{ \pm \}$. Thus $\Gamma =
\langle \sigma_\alpha; \alpha \in \Delta \rangle$, where
$\sigma_\alpha(\omega_\beta) = (-1)^{\delta_{\alpha, \beta}}$ and
the $\omega_\beta$'s are the fundamental weights, $\beta \in
\Delta$. $\Gamma$ has a natural action of $\W$, so we can form
$\tW := \Gamma \rtimes \W$. We consider the following action of
$\tW$ on $\C\Lambda$: the subgroup $\W$ act by the
$\bullet$-action and $\sigma \in \Gamma$ act by $\sigma(K_\lambda)
= \sigma(\lambda)K_{\lambda}$, for $\lambda \in \Lambda$. Let
$\C\Lambda^{\tW}$ be the invariant ring. Observe that
$\C\Lambda^{\Gamma} = \C2\Lambda$ so that $\C\Lambda^{\tW} =
\C2\Lambda^\W$. There is the Harish-Chandra isomorphism
$$
\chi: \Z^{HC} \isoto \C\Lambda^{\tW}
$$
For $\lambda \in T_\Lambda$ let $\chi_\lambda: \Z^{HC} \to \C$ be
the corresponding central character. This construction is standard
when $\g$ is semi-simple. Our reductive $\g$ can be written as a
direct sum of Lie algebras: $\g = [\g, \g] \oplus \Z_\g$, where
$\Z_\g$ is the center of $\g$ and $[\g,\g]$ is semi-simple. Thus,
$\Z_\g \subset \h$ and we have $\Uq = \Uq([\g,\g]) \ot \Uq(\Z_\g)$
(where $\Uq(\Z_\g) \subset \C\Lambda$). The Harish-Chandra
homomorphism $\chi$ for $\Uq$ can thus be described as the product
$\chi = \chi_{[\g,\g]} \ot Id_{\Uq(\Z_\g)}$, where
$\chi_{[\g,\g]}$ is the Harish-Chandra isomorphism for the quantum
group $\Uq([\g,\g])$.

If $q$ is an $l$'th root of unity we have $\Z = \Z^{HC}
\ot_{\Z^{(l)} \cap \Z^{HC}} \Z^{(l)}$ and if $q$ is not a root of
unity we have $\Z = \Z^{HC}$.

Note that to describe $\Z^{HC}(\Uq(\lo))$ we should consider
$\tW_P = \Gamma_P \rtimes \W_P$ where $\Gamma_P = \langle
\sigma_\alpha; \alpha \in \Delta_P \rangle$. We get then the
Harish-Chandra isomorphism
$$
\chi_{\lo}: \Z^{HC}(\Uq(\lo)) \isoto \C\Lambda^{\tW_P}.
$$

\subsubsection{}
Part \textbf{i)} of the following lemma is standard and part
\textbf{ii)} is proved in \cite{BK10} for the enveloping algebra
case and the proof in the generic quantum case is the same.
\begin{lemma}\label{Pregularweightlemma}  Assume that $q$ is generic. Let
$\lambda \in T_\Lambda$. \textbf{i)}  Then $\lambda$ is dominant
iff for all $\mu \in \Lambda_{r +} \setminus \{0\}$ we have
$\chi_{\lambda + \mu} \neq \chi_\lambda$,

\smallskip

\noindent \textbf{ii)} Let $\lambda$ be $P$-regular and dominant.
Let $\mu$ be a $P$-character. Then for any $\psi \in
\Lambda(V_{G_q}(\mu)), \psi \neq \mu$, we have $\chi_{\lambda +
\mu} \neq \chi_{\lambda +\psi}$.
\end{lemma}

\subsection{Integral versions of $\Uq$}\label{Integral versions of Uq} \subsubsection{} Let $t$ be a parameter and let $\U_t$
be the $\C(t)$-algebra defined by the same generators as $\Uq$ and
modulo the relations obtained by substituting $t$ for $q$ in the
defining relations of $\Uq$. Let $\A = \mathbb{C}[t,t^{-1}]$. Let
$\UAres$ be Lusztig's integral form of $\Uq$, the $\A$-algebra in
$\U_t$ generated by divided powers $E^{(n)}_\alpha =
E^n_\alpha/{[n]_{d_\alpha}!}$, $F^{(n)}_\alpha =
F^n_\alpha/{[n]_{d_\alpha}!}$, $\alpha$ a simple root, $n \geq 1$
(where $[m]_d = {\prod^m_{s=1} {q^{d\cdot s} - q^{-d\cdot
s}}\over{q^{d} -  q^{-d}}}$) and the $K_\mu$'s, $\mu \in P$. There
is also the de Concini-Kac integral form $\UA$, which is generated
over $\A$ by the $E_\alpha, F_\alpha$ and $K_\mu$'s. The
subalgebra $\UA$ is preserved by the left (and right) adjoint
actions of $\UAres$. The braid operators $T_w$ preserve these
integral versions.

$\OA$ is defined to be the dual of $\UAres$. This is a Hopf
$\A$-subalgebra of $\Oq$.

Similarly we get integral versions $\UA(\p)$ and $\UAres(\p)$ that
are subalgebras of  $\UA$ and $\UAres$, respectively.

Specializing $t \mapsto q$ we get $\Uq(\p)$ and $\Uq^{\res}(\p)$
as well. For a generic $q$ we have $\Uq^{\res}(\p) = \Uq(\p)$.

\subsection{Integrable part of $\Uq$}\label{loc fin section}

\subsubsection{Integrability of modules} A (say right)
$\Uq^{\res}(\lo)$-module $M$ is called integrable if there is an
$L_q$-module structure on it such that $um = \langle u, m_1\rangle
m_2$, for $u \in \Uqres(\lo), m \in M$, $m_1 \ot m_2$ is the coaction
on $m$ and $\langle \ , \ \rangle: \Uqres(\lo) \ot \Oq(L) \to \C$ is the
natural pairing. For $q$ generic $M$ is integrable
iff the $\Uqres(\lo)$-action is locally finite and the $K_\mu$'s act by
integer eigenvalues. If $q$ is a root of unity and $M$ admits an $\A$-form $M_\A$ we have that $M$ is
integrable if the $\UAres(\lo)$-action on $M_\A$ is locally finite and
the $K_\mu$'s act by integer eigenvalues.

Any (right) $\Uq^{\res}(\lo)$-module $M$ has a unique maximal submodule
$M^\lf$ on which the $\Uq^{\res}(\lo)$ action is
integrable.
(On the other hand, an  $L_q$-module structure
always differentiates to a $\Uq(\lo)$-module structure.)

We write $\Uq^{\fin} := \Uq^{\g\text{-}\fin}$ and  $M^{\fin} := M^{\g\text{-}\fin}$, if $M$ is a ${\Uqres}$-module.

\subsubsection{}
$\UqP$ is a subalgebra and a left
coideal in $\Uq$, i.e. $\triangle \UqP \subset \UqP \ot \Uq$ (but
it is not a Hopf algebra).

$\Uqf$ was first systematically studied in \cite{JL92}. (They
called it the ``ad-finite" subalgebra, but since this is
misleading at a root of unity we prefer the name ``ad-integrable".
In \cite{BK06, BK08} we also called it $\Uq^{fin}$
instead of $\Uqf$.) Let $\omega_i$ and $\alpha_i$, $1 \leq i \leq
r := \rank \g$, be the fundamental weights and the simple roots,
respectively. To give the reader a feeling for $\Uqf$ we state:
\begin{lemma}\label{Its a big subalgebra!} $i)$ $K_{2\omega_i}, K_{2\omega_i - \alpha_i}E_{\alpha_i}$ and
$K_{2\omega_i}F_{\alpha_i} \in
\Uqf$. $ii)$ $\Uqf \ot_{\Uqf \cap \C\Lambda} \C\Lambda = \Uq$.
\end{lemma}
\begin{proof} A computation (inside $\UA$) shows that
$$ad^2_r(E_{\alpha_i})(K_{2\omega_i}) =
ad_r(E_{\alpha_j})(K_{2\omega_i})=
ad^2_r(F_{\alpha_i})(K_{2\omega_i}) =
ad_r(F_{\alpha_j})(K_{2\omega_i}) = 0, \; i \neq j.
$$
An application of the PBW-theorem shows that this implies that
$K_{2\omega_i} \in \Uqf$, for all $i$. Since $K_{2\omega_i -
\alpha_i}E_{\alpha_i} \thicksim ad_r(E_{\alpha_i})(K_{2\omega_i})$
and $K_{2\omega_i}F_{\alpha_i} \thicksim
ad_r({F_{\alpha_i}})(K_{2\omega_i})$ (where $\thicksim$ means
equal up to a non-zero scalar) we have proved $i)$. $ii)$ follows
from $i)$.
\end{proof}

We have $\ZHC(\Uq) \subseteq \Uqf$. Let $\ZHC(\Uqf) := \ZHC(\Uq)$
be the Harish-Chandra center of $\Uqf$. On the other hand
$\Z^{(l)}(\Uq) \nsubseteq \Uqf$, for $q$ an $l$'th root of unity.

\begin{example} Let $\Uq(\mathfrak{sl}_2) = \C \langle E,F,K,K^{-1}
\rangle$, where $KEK^{-1} = qK, KFK^{-1} = q^{-1}F$ and $[E,F] =
\frac{K^2-K^{-2}}{q-q^{-1}}$. Then $\Uq(\mathfrak{sl}_2)^{\fin} =
\C \langle K^2 F, K^2, E , z\rangle$, where $z =
\frac{qK^2+q^{-1}K^{-2}}{(q-q^{-1})^2} + F E \in
\Z(\Uq(\mathfrak{sl}_2))$ is the Casimir operator. Thus $\Uqf \cap
\C\Lambda = \C[K^2]$ and $\C \Lambda = \C[K,K^{-1}]$.
\end{example}

Another important feature is that, contrary to $\Uq$, $\Uqf$ is
free over its Harish-Chandra center, except possibly for a finite
set of roots of unity, see \cite{JL92, B00, BK11}. This freeness
property holds only for the simply connected version of $\Uq$
which is a main reason we work with that version.

We are primary interested in the representation theory of $\Uq$,
but it will be $\Uqf$ that occurs naturally as global sections,
see Theorem \ref{thmtag0}. However, we remark that the
representation theories of $\Uq$ and $\Uqf$ are very similar. This
will be precisely explained in the next section.

Following \cite{JL94} we define
\begin{definition}\label{Uql-def} Let $\lambda \in T_\Lambda$. Put $\Uql := \Uqf/(\Ker \chi_\lambda)$.
\end{definition}
The right $G_q$-action on $\Uqf$, resp. on $\Uql$, that is
obtained by integrating $ad_r(\Uq)$ is called the right adjoint
$G_q$-action and again denoted by $ad_r$.

\begin{rem} There was a misprint in the paper \cite{BK06} which unfortunately, partly, moved on to \cite{BK08}; there we
defined $\Uql$ to be $\Uqf/Ann_{\Uqf}(M_\lambda)$. At a generic
$q$ this is the same as the correct definition \ref{Uql-def} given
here but at a root of unity it is wrong.
\end{rem}
\subsection{Verma modules and universal Verma
modules}\label{Verma modules and universal Verma modules}
\subsubsection{} There is the Verma module $M_\lambda := \Uq \ot_{\Uq(\bo)}
\C_\lambda$ for $\Uq$ with highest weight $\lambda \in T_\Lambda$,
where $\C_\lambda$ is the 1-dimensional representation of
$\Uq(\bo)$ defined by $\lambda$. For $q$ generic we have the
\emph{quantum Duflo formula} (see \cite{JL94} and \cite{BK11})
$$
Ann_{\Uq^{\fin}}(M_\lambda) = \Uq^{\fin} \cdot \Ker \chi_\lambda.
$$
Let $\lambda \in T_\Lambda$ and let $\mu_\lambda$ be the highest
weight vector of $M_\lambda$. Denote by $M_\lambda |_{\Uqf}$ the
Verma module $M_\lambda$ considered as a module for the subalgebra
$\Uqf$ of $\Uq$. Restriction defines an algebra map $\phi:
T_\Lambda \to T_{\Lambda^{\fin}} := Maxspec \, (\C \Lambda \cap
\Uqf)$. We get
\begin{lemma}\label{Verma module theory} $\Uqf \cdot \mu_\lambda = M_\lambda$. Moreover, if
if $\lambda, \lambda' \in T_{\Lambda}$ satisfy $\phi(\lambda) =
\phi(\lambda')$ then $M_\lambda |_{\Uqf} \cong M_{\lambda'}
|_{\Uqf}$. Here $M_\lambda |_{\Uqf}$ denotes $M_\lambda$
considered as a module over $\Uqf$.
\end{lemma}
Remark that we could also have considered Verma modules for $\Uqf$
to be parameterized by $T_{\Lambda^{\fin}}$. It is sometimes a
subtle issue which version of the quantum group one should use,
i.e. which functions on the tori one should include. We have
chosen to work with the simply connected version - because it is
free over $\ZHC$ - but note that \emph{all} versions of $\Uq$ have
the same Verma modules, as sets, and that they are parameterized
by the spectra of the torus part of the quantum group in question.
\begin{definition}\label{Defining universal Verma} Let $\wMPq := \UqP/\UqP \cdot \Uq(\uo)_{>0}$ be a
``$P_q$-universal" Verma module for $\UqP$. For $\lambda \in
T_\Lambda$ let $\MPql := \wMPq \ot_{\ZHC(\Uq(\lo))} \C_\lambda$.
\end{definition}
Here $\C_\lambda$ is the $1$-dimensional representation for
$\ZHC(\Uq(\lo))$ on which $z \in \ZHC(\Uq(\lo))$ acts by
$\chi_{\lo,\lambda}(z)$. Observe that the right adjoint action
$ad_r$ of $\Uq(\p)$ on $\wMPq$ integrates to a $P_q$-action. We
shall always consider $\wMPq$ with this $P_q$-action and in
particular its restricted $L_q$-action.

Observe that for $P=B$ we have that $\MBql$ coincides with the
usual Verma module $M_\lambda$.

By corollary \ref{finite parabolic deco} we see that the canonical
map $\Uq(\overline{\p})^\lf \to \wMPq$ is an isomorphism of
$L_q$-modules.

\subsubsection{} Denote by $\Lambda(V)$ the set of weights of a
$\Uq$-module $V$. We shall need a quantum version of a classical
result of Bernstein and Gelfand:
\begin{lemma}\label{central character lemma} Let $V$ be a finite
dimensional $\Uq^{\fin}$-module and assume that $M$ is a
$\Uq^{\fin}$-module such that $\Ker \chi_\lambda \cdot M = 0$.
Then $I \cdot M \ot V = 0$, where $I := \prod_{\mu \in \Lambda(V)}
\Ker \chi_{\lambda + \mu}$.
\end{lemma}
\begin{proof}
We shall prove \ref{Verma module tensor ann 2} only for the case
that $q$ is generic. The general case can be deduced from this,
using a continuity argument and integral forms of $\Uq, I$ and
$V$, but we have omitted the details. Its enough to prove that
\begin{equation}\label{Verma module tensor ann 2}
I \cdot \Uql \ot V = 0.
\end{equation}
Using a suitable $\Uq(\bo)^{\fin}$-filtration on $V$ with
1-dimensional subquotients a standard argument shows that
\begin{equation}\label{Verma module tensor ann}
I \cdot M_\lambda \ot V = 0.
\end{equation}
 By the quantum Duflo theorem we have that $M_{\lambda}$ is a
 faithful representation of $\Uql$. Thus we are in the position to rerun the argument from the proof of the enveloping
algebra case, Theorem 3.5 in \cite{BerGel}, to deduce that
\ref{Verma module tensor ann} implies \ref{Verma module tensor ann
2} in this case.
\end{proof}

\section{Parabolic quantum flag manifold}
In \cite{BK06} a quantum flag manifold, or more precisely the
category of quasi-coherent sheaves on it, was defined. Here we use
the same method to quantize a parabolic flag manifold.
\subsection{Definition of the parabolic quantum flag manifold}\label{Definition of the quantum flag manifold}
\subsubsection{} Let $\triangle$ be the comultiplication on $\Oq$. The composition
\begin{equation}\label{e2}
\Oq \overset{\triangle}{\to} \Oq \otimes \Oq \to \Oq(P) \otimes
\Oq\end{equation} defines a left $\Oq(P)$-comodule structure on
$\Oq$.

\smallskip

\medskip

A $P_q$-equivariant sheaf on $G_q$ is a triple $(F, \alpha,
\beta)$ where $F$ is a vector space, $\alpha: \Oq \otimes F \to F$
a left $\Oq$-module action and $\beta: F \to \Oq(P) \otimes F$ a
left $\Oq(P)$-comodule action such that $\alpha$ is a left
comodule map, where we consider the diagonal comodule structure on
$\Oq \otimes F$. Morphisms of $P_q$-equivariant sheaves on $G_q$
are $\Oq$-linear and $P_q$-linear maps.
\begin{defi}\label{hejsan} We denote by $\MPG$ the category of $P_q$-equivariant
sheaves on $G_q$.
\end{defi} We shall refer to objects of $\MPG$ as
$(\Oq,P_q)$-modules. Classically, let $\Po = G/P$ be the parabolic
flag manifold. There is an equivalence $\MOD(\BGG(G), P) \cong
\MOD(\BGG_{\Po})$, where $\MOD(\BGG_{\Po})$ is the category of
quasi-coherent sheaves on $\Po$. For this reason we like to think
of objects of $\MPG$ as ``(quasi-coherent) sheaves on $\Po_q$".

\subsection{Vector bundles and line bundles}\label{Vector bundles and line bundles} \subsubsection{} Let $V$ be a finite
dimensional $P_q$-module. Then we have
$$
\BGG_q \ot V \in \MPG
$$
where $P_q$ acts (i.e. $\Oq(P)$ coacts) diagonally and $\Oq$ acts
on the first factor. Since $\Oq$ is a $G_q\text{-}P_q$-bimodule
(i.e. an $\Oq(P)\text{-}\Oq$-bicomodule) we see that $\BGG_q \ot
V$ comes with a left $G_q$-action as well. We can think of $\BGG_q
\ot V$ as a $G_q$-equivariant vector bundle on $(\Po)_q$.

When $\lambda$ is a $P_q$-character (in which case also $-\lambda$
is a $P_q$-character) we put $\C_{\lambda} := V_{P_q}(\lambda)$
for the corresponding one-dimensional representation and we denote
by
$$
\BGG_{\Po_q}(\lambda) := \Oq \ot \C_{-\lambda} \in \MPG
$$
the corresponding line bundle on $\Po_q$. We shall also use the
notation
$$
M(\lambda) := M \ot \C_{-\lambda} \in \MPG,
$$
for $M \in \MPG$ and $\lambda$ a $P_q$-character.

\subsection{Global sections, direct and inverse image}\label{global sections
Section} \subsubsection{} There is the \emph{global section
functor} $$\Gamma: \OPG \to \MOD(\C), \ M \mapsto M^{P_q}, \ M \in
\OPG.
$$
Let $Q \supseteq P$ be another parabolic subgroup of $G$
(actually, in this section we don't need that $Q$ and $P$ are
parabolic). Let $\Q = G/Q$ and think of a symbolic map $\pi_q:
{\Po}_q \to {\Q}_q$ as a ``quantization" of the projection
$\pi:\Po \to \Q$. Recall that we have the adjoint pair
$\ResQP:\MOD(Q_q)  \rightleftarrows \MOD(P_q): \IndPQ$. It induces an
adjoint pair of functors
\begin{equation}\label{inverse direct def}
\pQPus: \OQG \rightleftarrows \OPG:\pQPls.
\end{equation}
Here, $\pQPus M = \ResQP M$ as a $P_q$-module and the $\Oq$-module
structure is the given by that on $M$, for $M \in \OQG$.

Similarly, $\pQPls M = \IndPQ M$ as a $Q_q$-module and its
$\Oq$-module structure is as follows:

Let $\triangle': \Oq \to \Oq(Q) \ot \Oq$ be the left coaction of
$\Oq(Q)$ on $\Oq$ that is obtained by integrating the action
$\mu_r$ of $\Uq(Lie \, Q)$ on $\Oq$. Let $f \ot m \in \pQPls M$
and let $\alpha \in \Oq$. Write $\triangle' \alpha = \alpha_1 \ot
\alpha_2$. Then we have $\alpha \cdot (f \ot m) = \alpha_1f \ot
\alpha_2 m.$

\medskip

\noindent We have that $\pQPus$ is exact and $\pQPls$ is right
adjoint to $\pQPus$; thus $\pQPls$ is left exact and maps
injectives to injectives. Let
$$\Theta: \OGG \isoto \MOD(\C), \, V
\mapsto V^G,$$ be the canonical equivalence. There is a natural
equivalence of functors $\Theta \circ \pGPls \cong \Gamma$.

\begin{lemma}\label{enough injectives lemma} $\OPG$ has enough injectives, hence the derived functors
$R\Gamma$ and $R\pQPls$ are well-defined.
\end{lemma}
\begin{proof}
Let $p: G \to \Po$  be the projection, let $M \in \OPG$ and let
$p^*_q M \to I$ be an injection of $p^*_q M$ into an injective
object $I \in \MOD(\Oq)$. Then $p_{q*} I$ is injective and we get an
injective composition
$$
M \to p_{q*} p^{*}_{q} M \to p_{q*} I.
$$
\end{proof}
Note that $R\pi^{\Q'q}_{\Q q*} \circ R\pQPls = R\pi^{\Q'q}_{\Po
q*}$, for $Q' \supseteq Q$ a third parabolic, since $R\pi_{q*}$
maps injectives to injectives.

\subsection{Ampleness of line bundles on $\Po_q$}
\subsubsection{} Let $\lambda \in T_\Lambda$ and fix $P$. We let $\lambda >>0$ mean
that $\lambda$ is a $P_q$-character and $\lambda(H_\alpha) >>0$
for each $\alpha \in \Delta-\Delta_P$. Our result here is
\begin{proposition}\label{aff} We have
\begin{enumerate}
 \item $R^{>0}\Gamma(\BGG_{\Po_q}(\lambda)) = 0$, for $\lambda \in
\Lambda_+$ a $P_q$-character.

\item The global section functor $\Gamma$ on $\MPG$ has finite
cohomological dimension.

\item Each object in $\MPG$ is a quotient of a direct sum of
$\BGG_{\Po_q}(\lambda)$'s.

\item Any surjection $M \twoheadrightarrow M'$ of noetherian (i.e.
$\Oq$-coherent) objects  in $\MPG$ induces a surjection
$\Gamma(M(\lambda)) \twoheadrightarrow \Gamma(M'(\lambda))$ for
$\lambda >> 0$.

\item If $M \in \MPG$ is noetherian, then
$R^{>0}\Gamma(M(\lambda)) = 0$ for $\lambda >>0$.
\end{enumerate}
\end{proposition}
\begin{proof} The case $P=B$ was dealt with in \cite{BK06}. We
shall reduce to that case.

\medskip

\noindent $(1)$ Let $\lambda \in \Lambda_+$ be a $P_q$-character.
We must prove that $R^{>0}\Gamma(\BGG_{\Po_q}(\lambda)) = 0$. By
the tensor identity we have $R \Gamma(\BGG_{\Po_q}(\lambda)) = \Oq
\ot RInd^{G_q}_{P_q}(\C_{-\lambda})$; hence, it is enough to prove
the \emph{Kempf vanishing}
\begin{equation}\label{Kempf P-vanishing}
R^{>0} Ind^{G_q}_{P_q}(\C_{-\lambda}) = 0.
\end{equation}
By \cite{APW},  \ref{Kempf P-vanishing} holds for $P= B$ (and any
$\lambda \in \Lambda_+$). Thus we have
$$
Ind^{G_q}_{B_q}(\C_{-\lambda}) = R Ind^{G_q}_{B_q}(\C_{-\lambda})
= R Ind^{G_q}_{P_q} \circ  R Ind^{P_q}_{B_q}(\C_{-\lambda}).
$$
Now, $R Ind^{P_q}_{B_q}(\C_{-\lambda}) = R Ind^{L_q}_{(L \cap
B)_q}(\C_{-\lambda}),$ with trivial $\Uq(\uo)$-action on the
latter. Since $L \cap B$ is a Borel subgroup of the reductive
group $L$, \cite{APW} applies again, so we get $R
Ind^{P_q}_{B_q}(\C_{-\lambda}) = Ind^{P_q}_{B_q}(\C_{-\lambda})$.
Since $\lambda=0$ on $\Delta_P$ it is clear that
$Ind^{P_q}_{B_q}(\C_{-\lambda}) = \C_{-\lambda}$. Thus,
$$
RInd^{G_q}_{P_q}(\C_{-\lambda}) = Ind^{G_q}_{B_q}(\C_{-\lambda}).
$$
Thus \ref{Kempf P-vanishing} holds and $1)$ is proved.

\medskip

\noindent $(2)$ Let $M \in \MPG$. Then we have from the tensor
identity and the result of $(1)$ applied to $\C_0$ that
$$
 RInd^{P_q}_{B_q}(M) = M \ot RInd^{P_q}_{B_q}(\C_0) = M.
$$
Thus
$$
RInd^{G_q}_{B_q}(M) = RInd^{G_q}_{P_q}\circ RInd^{P_q}_{B_q}(M) =
RInd^{G_q}_{P_q}(M),
$$
which proves that $RInd^{G_q}_{P_q}$ has finite cohomological
dimension, since we know from \cite{APW} that $RInd^{G_q}_{B_q}$
has cohomological dimension $\leq \dim \B$.

\medskip

\noindent $(3),(4)$ and $(5)$ can now formally be deduced from
$(1), (2)$ and Lemma \ref{enough injectives lemma} by the same
arguments as those in \cite{BK06}.
\end{proof}
\begin{rem} Using the multi-graded version, Proposition 2.1, in
\cite{BK06}, of a theorem of Artin and Zhang about non-commutative
projective schemes one can deduce that $\MPG \cong
\mathbf{Proj}(A_q)$, where $A_q$ is the ring  $\oplus_{\lambda
\in P-characters} \Gamma (\BGG_{\Po q}(\lambda))$.
\end{rem}

\section{Modules over extended quantum differential operators}
We define some algebras of quantum differential operators on $G_q$
and then we define categories of quantum $\D$-modules on $\Po_q$.
\subsection{Algebras of differential operators on $G_q$}
\subsubsection{} The construction given here is a version of the Heisenberg double,
see \cite{M93}. Recall the actions $\mu_l$ and $\mu_r$ of $\Uq$ on
$\Oq$ from \ref{23}, the left and right adjoints action $ad_l$ and
$ad_r$ of $\Uq$ on itself. In \cite{BK06} we defined the ring of
differential operators $\Dq$ on $G_q$ to be the smash product $\Dq
:= \Oq \# \Uq$ with respect to the action $\mu_l$.

Both the algebras $\Oq \ot \Uq$ and $\Dq$ are right $\Uq$-module
algebras with respect to the action  $\mu_r$ on $\Oq$ and the
action $ad_r$ on $\Uq$. We shall refer to these actions as the
right adjoint actions of $\Uq$ on $\Oq \ot \Uq$ and on $\Dq$,
respectively, and denote them by $ad_r$.

\medskip

\noindent The algebra $\Dq$ is suitable in relation to equivariant
sheaves of differential operators on $\B_q$, but it turns out that
for each parabolic $P$ it is better to use a different version of
it (see Remark \ref{explain it}). Since $\UqP$ is a left coideal
in $\Uq$ we can define subalgebras of $\Dq$ by
\begin{definition} Let $\DqP = \Oq \# \UqP$.
\end{definition}
Observe that this coincide with our earlier definition: $\DqP$ =
$ad_r(\Uq(\lo))$-integrable part of $\Dq$. Note that $\DqP = \Dq
\iff \lo = \h$. We write $\Dq^{\fin} = \Dq^{\g\text{-}\fin}$.
$\DqP$ is a $\Uq$-submodule algebra of $\Dq$.

\medskip

\noindent The action $ad_r |_{\Uqf}$ integrates to an action
$$
coad: \Uqf \to \Oq \ot \Uqf.
$$
This makes $\Uqf$ an $\Oq$-comodule algebra, i.e. $coad$ is an
algebra homomorphism. From the tensor identity we have $\Uqf \cong
coad(\Uqf) = (\Oq \ot \Uqf)^{G_q}$. $coad$ is however not the
embedding that we are primary interested in. We have
\begin{lemma}\label{right invariant embedding lemma}
There is an injective algebra homomorphism $\epsilon_l:
(\Uqf)^{op} \to \Dq^{\fin}$ whose image is the space of right
$G_q$-invariants $(\Dq^{\fin})^{G_q}$.\footnote{Recall that in
classical Lie theory differentiating the \emph{left} action of $G$
on itself gives an embedding of $\g$ into
\emph{right} invariant vector fields on $G$.}
\end{lemma}
\begin{proof} Let
$\U^{\fin}_{q,\smallleft}$ be the integrable part for the left
adjoint action. We have the algebra homomorphism $S: (\Uqf)^{op}
\to \U^{\fin}_{q,\smallleft}$. Let $\epsilon'_l$ be defined as the
composition
$$
\U^{\fin}_{q,\smallleft} \overset{coad}{\longrightarrow} \Oq \ot
\U^{\fin}_{q,\smallleft} \overset{1 \ot S^{-1}}{\longrightarrow}
\Oq \ot \U^{\fin}_{q}  {=} \Dq^{\fin}.
$$
Put $\epsilon_l = S \circ \epsilon'_l$. It follows from \cite{M93}
that $\epsilon_l$ is an algebra embedding.
\end{proof}
Observe that the $\epsilon_l$ does \emph{not} extend to an
embedding $\UqP \to \Dq$, unless $\lo = \g$.

\medskip

\noindent There is also the left adjoint action of $\Uq$ on $\Dq$.
This action commutes with the right adjoint action and is defined
by taking the action $\mu_l$ on $\Oq$ and the trivial $\Uq$-action
on itself. We denote this action by $ad_l$. It integrates to a
$G_q$-action. It restricts to a $G_q$-action on $\DqP$. We get the
embedding
$$\epsilon_r: \UqP \cong 1 \# \UqP
\hookrightarrow \DqP.$$ Thus $\epsilon_r(\UqP) = {}^{G_q}\DqP$,
equal the space of invariants for the left adjoint $G_q$-action.
Note that $\ZHC(\Uq(\lo)) \subseteq \UqP$ so that
$\epsilon_r(\ZHC(\Uq(\lo))) \subset \DqP$.

\subsection{Definition of quantum $\D$-modules on $\Po_q$}\label{Definition of quantum D-modules}
\subsubsection{} \begin{defi}\label{Dmoddef} Let $\DPG$ be the category whose object $M$ satisfies
\begin{enumerate}[(1)]
\item  $M$ is a left $\DqP$-module.
\end{enumerate}
 \begin{enumerate}[(2)]
 \item  $M$ has a right
$P_q$-action $\rho$ such that $\DqP \otimes M \to M$ is
$\Uq(\p)$-linear.
\end{enumerate}
 \begin{enumerate}[(3)]\item $d\rho(x)m = \epsilon_r(x)m$, for
 $x \in \Uq(\uo)$, $m \in M$.
\end{enumerate}
\end{defi}
The $\Uq(\p)$-linearity in $ii)$ means with respect to the action $d\rho$ on $M$ and the right $\Uq(\p)$-action on $\DqP \ot M$ given by
$(y \ot m) \cdot x = ad_r(x_2)(y) \ot d\rho(x_1)(m)$, for $y \ot m \in \DqP \ot M$ and $x \in \Uq(\p)$.
Morphisms are $\DqP$-linear and $P_q$-linear maps. We define the
global section functor $\Gamma$ on this category to be the functor
of taking $P_q$-invariants. Recall the $P_q$-universal Verma
module $\wMPq$ from Section \ref{Verma modules and universal Verma
modules}. Let
$$\DtP := \Oq \ot
\epsilon_r(\wMPq) \in \DPG.$$ Observe that $\DtP$ represents the
global section functor on this category.

\begin{rems}\label{explain it} A) It is enough to verify condition $(3)$
on a set of $\DqP$-module generators of $M$. The reason for this
is that $\Uq(\uo)$ is a left coideal. Indeed, if $m \in M$ satisfy
$d\rho(x)m = \epsilon_r(x)m$ for all $x \in \Uq(\uo)$, then for $y
\in \DqP$ we also have
$$d\rho(x)(ym) = ad_r(x_2)(y)d\rho(x_1)m =
ad_r(x_2)(y)\epsilon_r(x_1)m = \epsilon_r(x)(ym)
$$
since $x_1 \in \Uq(\uo)$. This implies that if $M$ satifies
$(1)$ and $(2)$ it has a maximal subobject $M' := \{m \in M; d\rho(x)m = \epsilon_r(x)m, x \in \Uq(\uo), m \in M \}$ and a maximal quotient
object $M^{''} := M/( \DqP \cdot \{ d\rho(x)m - \epsilon_r(x)m, x \in \Uq(\uo), m \in M \})$
that satisfy $(1)-(3)$.

\medskip

\noindent B) The principal reason why we work with $\Uq(\lo)$-integrable
differential operators, rather than $\Uq$-integrable ones, is the
existence of the parabolic triangular decomposition of Corollary
\ref{finite parabolic deco}, which is crucial to understand the
structure of $\wMPq$ and hence that of $\DtP$.

Thus, by Remark \ref{rems} C), we must in the case $\p = \bo$ use
the full algebra $\Dq$. Then the $\Uq(\lo)$-integrability
conditions are naturally imposed since we want Theorem
\ref{thmtag0} to hold.
\end{rems}
\subsubsection{Action of $\ZHC(\Uq(\lo))$ on $\DPG$.} Consider
now the smash product algebra $\DqP \# \Uq(\lo)^{\fin}$. Note that
any $M \in \DPG$ has a natural left $\DqP \#
\Uq(\lo)^{\fin}$-action (the $\DqP$-action is the given one and
the $\Uq(\lo)^{\fin}$-action is given by $d\rho
|_{\Uq(\lo)^{\fin}}$).

By Lemma \ref{ast untwist lemma} and Corollary \ref{ast untwist
cor} we have
\begin{proposition}\label{alphatilde prop}
There is an algebra homomorphism
$$
\wtalp: \Uq(\lo)^{\fin} \to \DqP \# \Uq(\lo)^{\fin}, \ \wtalp(u) =
S(u_1) \ot u_2, u \in \Uq(\lo)^{\fin}.
$$
$\Jm \wtalp$ commutes with $\Dq \ot 1$, $\wtalp$ induces an
algebra isomorphism $$1_{\DqP} \ot \wtalp: \DqP \ot
\Uq(\lo)^{\fin} \overset{\sim}{\to} \DqP \# \Uq(\lo)^{\fin}$$ and
$\wtalp$ restricts to an embedding
$$
\alp := \wtalp |_{\ZHC(\Uq(\lo))}: \ZHC(\Uq(\lo)) =
\ZHC(\Uq(\lo)^{\fin}) \to \Z(\DqP \# \Uq(\lo)^{\fin}).
$$
\end{proposition}
Note that if $q$ is generic then $\Z(\DqP) = \C$ and $\alp$ is an
isomorphism.

\begin{defi}\footnote{We used a slightly different parametrization of the $\D$-module blocks in the enveloping algebra case,
compare \cite{BK10}, Proposition 4.4.} Let $\lambda \in T_\Lambda$ and let $\DPGl$ be
the category whose object $M$ satisfies $(1)-(3)$ and also
\begin{enumerate}[(4)]
 \item $(\alpha_{\lo}(z) - \chi_{\lo,\lambda}
 (z))m = 0, \; m \in M, z \in \ZHC(\Uq(\lo))$.
\end{enumerate}
Similarly, we let $\DPGwl$ be the category whose object $M$
satisfies $(1)-(3)$ and also
\begin{enumerate}[($\widehat{4}$)]
 \item  $\alpha_{\lo}(z) - \chi_{\lo,\lambda}
 (z)
\hbox{ is locally nilpotent on } M, \hbox{ for } z \in
\ZHC(\Uq(\lo)).$
\end{enumerate}
\end{defi}
\medskip

Again, the global section functor $\Gamma$ on $\DPGl$ is defined
to be the functor of taking $P_q$-invariants. Note that the object
$$\DPl := \Oq \ot
\epsilon_r({M}_{P_q,\lambda}) \in \DPGl$$ represents global
sections on this category.

\subsubsection{} Note that if $M \in \DPG$ and $V$ is a $P_q$-module such
that if we differentiate the $P_q$-action $\Uq(\uo) \subset
\Uq(\p)$ acts trivially on $V$, then we naturally have $M \ot V
\in \DPG$, by letting $P_q$ act diagonally and $\Dq$ on the first
factor. We get
\begin{Lem}\label{Lemacita} Let $\lambda \in T_\Lambda$, $M \in \DPGl$ and let
$V_{P_q}(\mu)$ be an irreducible $P_q$-module with highest weight
$\mu$. Then $M\ot  V_{P_q}(\mu)  \in \oplus_{ \nu \in
\Lambda(V_P(\mu))} \MOD(\Dq,P_q,\Uq(\uo), {\widehat{{
\lambda+\nu}}})$, where $\Lambda(V_{P_q}(\mu))$ denotes
the set of weights of $V_{P_q}(\mu)$.
\end{Lem}
\begin{proof}
We know that any object of $\DPG$ has a natural action of $\DqP \#
\Uq(\lo)^{\fin}$ and hence also of $\Uq(\lo)^{\fin}$ via the map
$\wtalp$ of Proposition \ref{alphatilde prop}. Let us refer to
this $\Uq(\lo)^{\fin}$-action as the $\wtalp$-action.

We observe that the $\wtalp$-action on $M \ot V_{P_q}(\mu)$ is the
diagonal action of the $\wtalp$-action on $M$ and the given
$\Uq(\lo)^{\fin}$-action on $V_{P_q}(\mu)$, which is obtained by
differentiating the given $L_q$-action. We have by assumption that
$\alpha_{\lo}(z) = \wt{\alpha}_{\lo}(z)$ acts by
$\chi_{\lo,\lambda}(z)$ on $M$, for $z \in \ZHC(\Uq(\lo))$.
Therefore the assertion of the lemma follows from Lemma
\ref{central character lemma}.
\end{proof}
\subsection{Direct and inverse image of $\D_q$-modules}\label{Dir Im}
\subsubsection{} For $P \subseteq Q$ parabolics recall that we have the adjoint pair
$(\pQPus, \pQPls)$ at the level of equivariant $\Oq$-modules (see
Section \ref{global sections Section}). We shall now construct
direct and inverse image functors between our $\Dq$-module
categories.

Let $L'$ and $R'$ be the Levi and the unipotent radical of $Q$ and
let $\lo'$ and $\uo'$ be their respective Lie algebras. Let $\qo$
be the Lie algebra of $Q$.

We shall construct a direct image functor
$$
\pQPlb: \DPG \to \DQG
$$
Let $M \in \DPG$. We define the underlying $(\Oq,Q_q)$-module of
$\pQPlb M$ to be $\pQPls M = \IndPQ M $. It remains to construct
an action of $\UqQ$ on $\IndPQ M$ satisfying certain
compatibilities.

Let for now $\triangle: \Uqf \to \Oq(Q) \ot \Uqf$ denote the left
coadjoint action of $\Oq(Q)$ on $\Uqf$ (i.e. the coaction obtained
by integrating the action $ad_r$ of $\Uq(\qo)$ on $\Uqf$). Since $
\triangle$ makes $\Uqf$ an $\Oq(Q)$-comodule algebra we have that
$ \triangle$ is an algebra homomorphism. By the tensor identity $
\triangle$ maps $\Uqf$ isomorphically onto the subspace $\IndPQ
\Uqf \subset \Oq(Q) \ot \Uqf$.

Moreover, the tensor identity provides an isomorphism
$$\Uqf \ot
\IndPQ M \to \IndPQ \Uqf \ot M$$ given by $u \ot (a \ot m) \mapsto
u_1 a  \ot u_2 \ot m$, for $a \ot m \in \IndPQ M$ and $u \in
\Uqf$, $ \triangle u = u_1 \ot u_2$. Composing this isomorphism
with the map $\IndPQ \Uqf \ot M \to Ind^{Q_q}_{P_q} M$, that is
given by the action map $\Uqf \ot M \to M$, we get a map
\begin{equation}\label{Uqf-structure on Ind}
\Uqf \ot \IndPQ M \to \IndPQ  M, \ u \ot (a\ot m) \mapsto u_1a \ot
u_2m.
\end{equation}
This gives a $\Uqf$-module structure on $\IndPQ M$, since $
\triangle$ is an algebra map.
\begin{proposition} $\pQPlb$ defines a functor $\DPG \to \DQG$.
\end{proposition}
\begin{proof} Let $M \in \DPG$. The $\Uqf$-action on $\pQPlb M$ is compatible with the $\Oq$-action so
that $\IndPQ M$ becomes a $\Dqf$-module. Moreover, this
$\Dqf$-action is $Q_q$-equivariant. In order to make $\pQPlb M$ a
$\DqQ$-module we shall use that \ref{general equality} provides us
with a $Q_q$-linear isomorphism:
$$
\pQPlb M \cong Ind^{L'_q}_{(P \cap L')_q} M.
$$
We transport the $\Oq$-action to the RHS making $Ind^{L'_q}_{(P
\cap L')_q} M$ an object of $\OQG$. In analogy with the above we
can equip $Ind^{L'_q}_{(P \cap L')_q} M$ with an $\UqQ$-structure
by the composition
$$
\UqQ \ot Ind^{L'_q}_{(P \cap L')_q} M \cong  Ind^{L'_q}_{(P \cap
L')_q} \UqQ \ot M \to Ind^{L'_q}_{(P \cap L')_q} M
$$
where the first isomorphism is the tensor identity and the second
map is induced from the action map $\UqQ \ot M \to M$. (This
extends the $\Uqf$-action previously defined.) Again, this
$\UqQ$-action is compatible with the $\Oq$ and $Q_q$-actions
making $\pQPlb M \cong Ind^{L'_q}_{(P \cap L')_q} M$ an object of
$\DQG$.
\end{proof}
Define a functor
\begin{equation}\label{inverse image of D-modules}
{\pQPub}: \DQG \to \DPG
\end{equation}
by ${\pQPub}(V) :=$ maximal $(\Dq, P_q)$-module quotient of $V$ on
which the two actions of $\Uq(\uo)$ coincide. The $P_q$-action on
${\pQPub}(V)$ is by definition the restriction of the
$Q_q$-action.

We observe that the forgetful functor $for: \DPG \to \MOD(\DqP)$ has
a right adjoint. From this it follows that $\DPG$ has enough
injectives. It is straightforward to verify that

\begin{proposition} The functor ${\pQPub}$ is right exact.
There is an adjoint pair of functors $$\pQPub: \DQG
\rightleftarrows \DPG: {\pQPlb}.$$
\end{proposition}
Note that the forgetful functors $$\DPG {\longrightarrow} \OPG
{\longrightarrow} \MOD(P_q)$$ map injectives to injectives. Thus the
derived functors $R{\pQPlb}$ can be computed as $R\IndPQ$ of
underlying $P_q$-modules.

\subsection{Description of global sections}
\subsubsection{} We keep the notations of Section \ref{Dir Im}. The main result of
this section is
\begin{Thm}\label{thmtag0}  For any $q \in \C^*$ except roots of unity of order smaller than
the Coxeter number of $\g$ we have $i)$ $R{{\pPBlb}}\DtB = \DtP
\ot _{ \C\Lambda^{\tW_P}} \C\Lambda$, $ii)$  $R{\pQPlb}(\DtP) =
\DtQ \ot _{\C\Lambda^{\tW_Q}} \C\Lambda^{\tW_P}$, $iii)$
$R\Gamma(\DtP) = \Uq^{\fin}
\ot_{\C\Lambda^{\tW}}\C\Lambda^{\tW_P}$ and $iv)$ $R\Gamma(\DtP)
\cong \tU$.
\end{Thm}
Part $iv)$ of this theorem, for the case $P=B$, was first proved
in \cite{BK06,BK08}. A new proof of $iii)$ and $iv)$ in the case
$P=B$ is given in \cite{BK11}. The strategy here is to reduce to that
case. Because of our usual equivalence $\MOD(\C) \cong
\MOD(\Oq,G_q)$ we see that a special case of $iii)$ implies that
$i)$ holds when $P=G$.
\begin{proof}[Proof of Theorem \ref{thmtag0}.]
\noindent \emph{Step a)} We have ${\pPBub} \DtP = \DtB$. By
adjunction this gives a morphism
$$
\phi_{B_q,P_q}: \DtP \to R{{\pPBlb}}\DtB.
$$
The embedding $\C\Lambda \to \wMBq$ induces a map $\alpha:
\C\Lambda \to R{{\pPBlb}}\DtB $. We observe that $\alpha
|_{\C\Lambda^{\tW_P}}$ coincides with the composition
$$\C\Lambda^{\tW_P} \overset{\chi^{-1}_{\lo,
\lambda}}\longrightarrow \ZHC(\Uq(\lo)) \hookrightarrow \UqP
\overset{\epsilon_r}{\to} \DqP \to \DtP
\overset{\phi_{B_q,P_q}}{\to} R{{\pPBlb}}\DtB.$$ We thus get a map
$$
\phi_{B_q,P_q} \ot \alpha: \DtP \ot_{{\C\Lambda}^{\tW_P}}
\C\Lambda \to R{{\pPBlb}}\DtB.
$$
By the tensor identity and lemma \ref{general equality} we have
$${R\pPBlb} \DtB \cong \Oq(P) \ot R\IndBP \wMBq \cong \Oq(P) \ot R\IndBL \wMBq.$$ Therefore, $\phi_{B_q,P_q} = 1_{\Oq(P)} \ot \overline{\phi}_{B_q,P_q}$
where
$$
\overline{\phi}_{B_q,P_q}: \wMPq \to R\IndBL \wMBq
$$
is given by $\overline{\phi}_{B_q,P_q}(m) = a_1 \ot
\overline{m}_2$, where $\Delta m = a_1 \ot m_2 \in \Oq(L) \ot
\wMPq$ is the coaction and $\overline{m}_2$ is the image of $m_2$
in $\wMBq$. There is also a natural map $\overline{\alpha}:
\C\Lambda^{\tW_P} \to \wMPq$ (obtained by restricting the image of
$\alpha$). Thus, again, we get a map
\begin{equation}\label{ind morphism nr 2c}
\overline{\phi}_{B_q,P_q} \ot \overline{\alpha}: \wMPq
\ot_{{\C\Lambda}^{\tW_P}} \C\Lambda  \to  R\IndBL \wMBq.
\end{equation}
Recall that by \cite{BK11} $\overline{\phi}_{B_q,G_q} \ot
\overline{\alpha}$ is an isomorphisms for all $q$ (except roots of
unity of order smaller than the Coxeter number of $\g$).

\medskip

\noindent \emph{Step b)} By corollary \ref{finite parabolic deco}
we have the isomorphisms $(*) \ \Uq(\uobar) \ot {\wt M}_{L_q}
\overset{\sim}{\to} \wMPq$ and $(**) \ \Uq(\uobar) \ot {\wt M}_{(B
\cap L)_q} \overset{\sim}{\to} \wMBq$, where ${\wt M}_{(B \cap
L)_q} := \Uq(\lo) \ot_{\Uq(\lo \cap \n)} \C$ is the universal
Verma module for $\Uq(\lo)$.
%

Under the isomorphisms $(*)$ and $(**)$ we see that the map
$\overline{\phi}_{B_q,P_q} \ot \overline{\alpha}$ corresponds to
$$
f: \Uq(\uobar) \ot \Uq(\lo)\ot_{{\C\Lambda}^{\tW_P}} \C\Lambda \to
R\IndBL \Uq(\uobar) \ot {\wt M}_{(B \cap L)_q},
$$
where $f(x \ot v \ot z) = x_1 v_1 \ot x_2 \ot \overline{v_2}\cdot
\overline{\alpha}(z)$, where $v \mapsto \overline{v}$ is the
natural projection $\Uq(\lo)^\fin \to {\wt M}_{(B \cap L)_q}$.

Again, (by \cite{BK11}) the map $ \Uq(\lo)^\fin
\ot_{{\C\Lambda}^{\tW_P}} \C\Lambda \to R\IndBL {\wt M}_{(B \cap
L)_q}$ given by $v \ot z \mapsto v_1 \ot \overline{v_2}\cdot
\overline{\alpha}(z)$ is an isomorphism. Thus also $f$ is an
isomorphism by the tensor identity. Thus $\phi_{B_q,P_q} \ot
\overline{\alpha}$ is an isomorphism. This proves $i)$.

\medskip

\noindent \emph{Step c)} Just like in \emph{Step a)} we get a map
\begin{equation}\label{ind morphism nr 3}
\overline{\phi}_{P_q,Q_q}: \wMQq \ot_{{\C\Lambda}^{\tW_Q}}
\C\Lambda^{\tW_P}\to R\IndPQ \wMPq
\end{equation}
that we must prove is an isomorphism. We have $R\IndBP \wMBq =
\wMPq \ot_{{\C\Lambda}^{\tW_P}} \C\Lambda$, so that
$$
(R\IndPQ \wMPq) \ot_{{\C\Lambda}^{\tW_P}} \C\Lambda = R\IndPQ
\wMPq \ot_{{\C\Lambda}^{\tW_P}} \C\Lambda = R\IndBQ \wMBq = \wMQq
\ot_{{\C\Lambda}^{\tW_Q}} \C\Lambda.
$$
Since $\C \Lambda$ is faithfully flat over $\C\Lambda^{\tW_P}$ it
follows that $\overline{\phi}_{P_q,Q_q}$ and hence
$\phi_{P_q,Q_q}$ is an isomorphism. This shows $ii)$. $iii)$
follows from a special case of $ii)$ by taking $G_q$-invariants.
Finally, $iv)$ is deduced from $iii)$ by specializing to
$\lambda$.
\end{proof}

\subsubsection{Localization functor.} Because of Theorem
\ref{thmtag0} the global section functor $\Gamma$ takes values in
certain categories of $\Uq$-modules:
$$\Gamma: \DPGl \to \MOD(\Uql) \hbox{ and }$$
$$\Gamma: \DPGwl \to
\MOD_{\widehat{\lambda}}(\Uq).$$ It is easy to see that both functors
have left adjoints, denoted by $\Loc$, which we call localization
functors. In the first case it is given by
$$
\Loc = \DPl \otimes_{\Uql} ( \ ): \MOD(\Uql) \to \DPGl
$$
and in the second case it is given by
$$
\Loc = \underleftarrow{\lim}_n (\DtP/(1 \ot \Ker \chi_\lambda)^n)
\otimes_{\Uq} ( \ ): \MOD_{\widehat{\lambda}}(\Uq) \to \DPGwl.
$$

\section{Singular Localization}
\subsubsection{} Here we prove the singular version of Beilinson-Bernstein
localization.
\begin{theorem}\label{singlocthmtag0} Let $q$ be generic and let $\lambda$ be dominant and $P$-regular. Then $$\Gamma: \DPGl \to
\MOD(\Uql)$$
is an equivalence of categories. \end{theorem}
\begin{proof}
Essentially taken from \cite{BB81}. Since $\Gamma(\DPl) = \Uql$,
which is a generator of the target category, the theorem will
follow from the following two claims:

$a)$ Let $\lambda$ be dominant. Then $\Gamma:  \DPGl \to \MOD(\Uql)$
is exact.

$b)$ Let $\lambda$ be dominant and $P$-regular and $M \in  \DPGl$,
then, if $\Gamma(M) = 0$, it follows that $M = 0$.

Let $V$ be a finite dimensional irreducible $G_q$-module and let
$$
0 = V_{-1} \subset V_{0} \subset \ldots \subset V_n = V
$$
be a filtration of $V$ by $P_q$-submodules, such that $V_i/V_{i-1}
\cong V_P(\mu_i)$ is an irreducible $P_q$-module. $M \ot V$ is a sheaf, i.e. an object of $\MOD(\Oq,P_q)$, equipped with the diagonal $\Uq$-action. Recall that $M
\ot V_P(\mu_i) \in \DPG$.

Assume that the highest weight $\mu_0$ of $V$ is a $P_q$-character. Then $V_0 = V_{P_q}(\mu_0) = \C_{\mu_0}$
 and we have $M \ot V_0 = M(-\mu_0)$ (see Section \ref{Vector bundles and line bundles} for these notations). Thus we get an embedding $M(-\mu_0)
\hookrightarrow M \ot V$, which twists to the embedding
$$M
\hookrightarrow (M \ot V)(\mu_0) \cong M^{\dim V}(\mu_0) = M(\mu_0)^{\dim V}.$$ Now, by
Lemmas \ref{Pregularweightlemma} i), \ref{Lemacita} and Theorem
\ref{thmtag0} iii) we get that this inclusion splits on derived
global sections, so $R\Gamma(M)$ is a direct summand of
$R\Gamma(M(\mu_0))^{\dim V}$. Now, for $\mu_0$ big enough and
$M$ $\Oq$-coherent we have $R^{>0}\Gamma(M(\mu_0)) = 0$, by
Proposition \ref{aff}. Hence, $R^{>0}\Gamma(M) = 0$ in this case.
A general $M$ is the union of coherent submodules and by a
standard limit-argument it follows that $R^{>0}\Gamma(M) = 0$.
This proves $a)$.

Now, for $b)$ we assume instead that the lowest weight $\mu_n$ of
$V$ is a $P$-character. Then we have a surjection $M^{\dim V}
\cong M\otimes V \to M(-\mu_n)$. Applying global sections and
using Lemmas \ref{Pregularweightlemma} ii), \ref{Lemacita} and
Theorem \ref{thmtag0} iv) we get that $\Gamma(M(-\mu_n))$ is a
direct summand of $\Gamma(M)^{\dim V}$. For $\mu_n$ small enough
we get that $\Gamma(M(-\mu_n)) \neq 0$. Hence, $\Gamma(M) \neq 0$.
This proves $b)$.
\end{proof}

\begin{theorem}\label{singlocthmtagGeneralized0} Let $q$ be generic and let $\lambda$ be dominant and $P$-regular then
$$\Gamma: \DPGwl \to \MOD_{\widehat{\lambda}}(\Uq)$$ is an equivalence
of categories.
\end{theorem}
\begin{proof} This follows from Theorem \ref{singlocthmtag0} and
a simple devissage.
\end{proof}

\section{Appendix}

\subsection{Hopf algebras}\label{appendix} \subsubsection{} For general information we refer to \cite{M93}.
Let $H$ be a Hopf algebra over a commutative ring. We denote by
$\mu$, $\triangle$, $S$, $\iota$ and $\epsilon$ the product,
coproduct, antipode, unit and counit, respectively, on $H$. We
shall use Sweedler's notation and write $\triangle x = x_1 \ot
x_2$ for the coproduct of $x \in H$.

If $M$ is a right $H$-comodule we denote by $\triangle: M \to M
\ot H$ the coaction and write $\triangle m = m_1 \ot x_2$, for $m
\in M$. If $N$ is another right $H$-comodule we have the
\emph{diagonal} coaction of $H$ on $M \ot N$ defined as the
composition
$$
M \ot N \overset{\triangle \ot \triangle}{\longrightarrow} (M \ot
H) \ot (N \ot H) \overset{F_{23}}{\longrightarrow} (M \ot N) \ot
(H \ot H) \overset{1 \ot \mu}{\longrightarrow} (M \ot N) \ot H,
$$
where $F_{23}$ flips the 2'nd and 3'rd tensor.

\medskip

\noindent Let $R$ be an algebra equipped with a (left) $H$-module
structure. $R$ is called a \emph{module algebra} for $H$ if $x(r
\cdot r') = x_1(r) \cdot x_2(r')$, for $x \in H$ and $r,r' \in R$.
We can then define the \emph{smash-product algebra} $R \# H$. As a
vector space $R \# H = R \ot H$ and its associative multiplication
is defined by
$$
r \ot x \cdot r' \ot x' = rx_1(r') \ot x_2x'.
$$

\subsubsection{Adjoint action}\label{finite part of Hopf algebra} The left adjoint action $ad_l$ of $H$ on itself is given by
$ad_l(x)(y) = x_1 y S(x_2)$. Similarly, there is the right adjoint
action $ad_r$ of $H$ on itself which is defined by $ad_r(x)(y) =
S(x_1)yx_2$. It makes $R$ a right $H$-module algebra.

\subsubsection{An untwisting lemma}\label{An untwisting lemma}
Assume that $H$ is isomorphic to a Hopf subalgebra of $R$ and
consider the action of $H$ on $R$ which is the restriction of the
left adjoint action of $R$ on itself. Then we have
\begin{lemma}\label{ast untwist lemma} There is an algebra homomorphism
$$
f: H \to R \# H, \ f(x) = S(x_1) \ot x_2, x \in H.
$$
Moreover, $\Jm f$ commutes with $R \ot 1$ and $f$ induces an
algebra isomorphism $1_{R} \ot f: R \ot H \overset{\sim}{\to} R \#
H$.
\end{lemma}
\begin{proof} For the first assertion, let $x,y \in H$.
Then
$$
f(x) \cdot f(y) = S(x_1) \ot x_2 \cdot S(y_1) \ot y_2 = S(x_1)
ad_l(x_{2})(S(y_1))\ot x_{3}y_2 =
$$
$$
S(x_{1}) ad_l(x_{2})(S(y_1)) \ot x_3y_2 =
S(x_{1})x_{2}S(y_1)S(x_{3}) \ot x_4y_2 =
$$
$$
\iota \circ \epsilon(x_{1})S(y_1)S(x_{2}) \ot x_3y_2 =
S(y_1)S(x_1) \ot x_1x_2 = f(xy).
$$
For the second assertion, let $r \ot 1 \in R \ot 1$ and $x \in H$.
Then
$$
f(x) \cdot r \ot 1 = S(x_1) ad_l(x_{2})(r) \ot x_{3} = S(x_{1})
x_{2} r  S(x_{3}) \ot x_{4} =
$$
$$
\iota \circ \epsilon(x_{1}) r S(x_{2}) \ot x_3 =rS(x_1) \ot x_2 =
r \ot 1 \cdot f(x).
$$
This implies that $1_{R} \ot f$ is an algebra homomorphism; its
inverse is given by $$R \# H \ni r \ot x \mapsto rx_1 \ot x_2 \in
R \ot H.$$
\end{proof}
Let $\Z(H)$ denote the center of (the underlying algebra of) $H$.
\begin{cor}\label{ast untwist cor} $f$ induces an algebra embedding $\overline{f}:
\Z(H) \to \Z(R \# H)$. $1 \ot f$ induces an isomorphism $\Z(R) \ot
\Z(H) \overset{\sim}{\to} \Z(R \# H)$. In particular, if $\Z(R) =
\C$, then $\overline{f}$ is an isomorphism.
\end{cor}

\subsection{Quantizing $\U(\uo)$}\label{Quantizing U(uo)} A canonical quantization $\Uq(\uo)$ was
defined in \cite{G11}. The following properties of it was proved
in \cite{G11} for a generic $q$. We shall prove them for any $q$
by modifying his methods.
\begin{proposition-definition}\label{Ur prop def} There are subalgebras $\Uq(\uo) \subseteq
\Uq(\bo)$ and $\Uq(\uobar) \subseteq \Uq(\overline{\bo})$ such
that the following holds:

\emph{i)} Multiplication define linear isomorphisms $\Uq(\uobar)
\ot \Uq(\lo) \isoto \Uq(\overline{\p})$, $\Uq(\lo) \ot \Uq(\uo)
\isoto \Uq({\p})$ and
 $\Uq(\uobar)
\ot \Uq(\lo) \ot \Uq(\uo) \isoto \Uq$ (parabolic triangular
decomposition).

\emph{ii)} $\Uq(\uo)$ and $\Uq(\uobar)$ are integrable
$ad_r(\Uq(\lo))$-modules.

\emph{iii)} $\Uq(\uo)$ is a left coideal in $\Uq(\bo)$.

\emph{iv)} $\Uq(\uo)$ specializes to $\U(\uo)$ at $q=1$.
\end{proposition-definition}
\begin{proof} The Majid-Radford theorem, \cite{Maj,
Rad}, implies the following: Let $\pi: H \to H_0$ be a split
projection of Hopf algebras (i.e. there exists a Hopf algebra
injection $\iota: H_0 \to H$ such that $\pi \circ \iota = Id$.)
Put $B := B(H,H_0) = \{x \in H; \pi(x_1) \ot x_2 = 1 \ot x\}$.
Then multiplication defines an isomorphism $H_0 \ot B \isoto H$.
Observe that $B$ is automatically stable under the right adjoint
action of $H_0$ on $H$ and that $B$ is a left coideal in $H$.

Note that if $H = \oplus_{n \in \mathbb{N}} H_{n}$ is an
$\mathbb{N}$-graded Hopf algebra, then the projection $\pi: H \to
H_0$ is split.

Assume that $H$ and $H_0$ are Hopf algebras over $\A$. Then we see
that the construction of $B$ above commutes with every
specialization $t \to q$. Because, if we let $B_q =
B(H_q,H_{0,q})$ and $B_{t \mapsto q} = B(H,H_0)_q$, we clearly
have $B_{t\mapsto q} \subseteq B_q$ and, since $H_{0q} \ot B_q =
H_q = H_{0q} \ot B_{t\mapsto q}$, we get $B_{t\mapsto q} = B_q$.

\medskip

Let $\UA(\p)'$ be the subalgebra of $\UAres(\p)$ generated by
$\UA(\p)$ and $\UAres(\lo)$. Consider on $\UA(\p)$ (resp. on
$\UA(\p)'$) the grading for which $\deg \Uq(\lo) = 0$ (resp. $\deg
\UAres(\lo) = 0$) and $\deg E_\beta = 1$, for $\beta \in \Delta
\setminus \Delta_P$. Let $B := B(\UA(\p), \UA(\lo))$ and $B' :=
B(\UA(\p)', \UAres(\lo))$. Since
$$
\UAres(\lo) \ot B' = \UAres(\p) = \UAres(\lo) \ot_{\UA(\lo)}
\UA(\p) =  \UAres(\lo) \ot_{\UA(\lo)} (\UA(\lo) \ot B) =
\UAres(\lo) \ot B.
$$
and, evidently, $B \subseteq B'$ we get $B=B'$. This implies that
$B$ is stable under the right adjoint action of $\UAres(\lo)$. We
shall next prove that this action is integrable:

Observe that $K_{-\beta}E_\beta \in B(\UA(\p), \UA(\lo))$. We have
(using the Serre relations) that
$$
ad^{1-\langle \alpha, \beta
\rangle}_r(E_\alpha)(K_{-\beta}E_\beta) =
ad_r(F_\alpha)(K_{-\beta}E_\beta) = 0, \ \alpha \neq \beta \in
\Delta.
$$
This implies that $K_{-\beta}E_\beta \in \UA^\lf$, for $\beta \in
\Delta \setminus \Delta_P$. Then $B$ is generated as an algebra by
the $\UAres(\p)$-module generated by $K_{-\beta}E_\beta$, $\beta
\in \Delta \setminus \Delta_P$. This follows from an induction
similar to the one given in the proof of Theorem 4.1 in
\cite{G11}. We have omitted the details here.

Thus we have proved that $\UAres(\lo)$-module structure integrates
to an $L_\A$-module structure on $B$. Putting $\Uq(\uo) := B_q$,
we get that $\Uq(\uo)$ is an $L_q$-module for which the first
isomorphism of $i)$ holds. Similarly, we construct an $L_q$-module
$\Uq(\uobar) \subset \Uq(\overline{\bo})$ such that the second
isomorphism of $i)$ holds. The third isomorphism of $i)$ follows
from the first two.

$ii)$ and $iii)$ are already proved. By a computation we have
$\U(\uo) = B(\U(\p), \U(\lo))$, which, together with the fact that
$B$ commutes with specializations, proves $iv)$.
\end{proof}
It follows from the constructions that $\Uq(\uo)$ and
$\Uq(\overline{\uo})$ are Hopf-algebras in the braided tensor
category of modules over the Drinfel'd double of $\Uq(\lo)$. But
they are not Hopf-subalgebras of $\Uq(\bo)$, resp. of
$\Uq(\overline{\bo})$, in the usual sense, i.e. they are not
closed under the coproduct, not even for $\uo = \n$.
\begin{cor}\label{finite parabolic deco}
Multiplication define linear isomorphisms $\Uq(\uobar) \ot
\Uq(\lo)^\fin \isoto \Uq(\overline{\p})^\lf$, $\Uq(\lo)^\fin \ot
\Uq(\uo) \isoto \Uq({\p})^\lf$ and
 $\Uq(\uobar)
\ot \Uq(\lo)^\fin \ot \Uq(\uo) \isoto \UqP$.
\end{cor}
\begin{proof} Let us prove the third isomorphism, the others are similar. We assume for simplicity that $q$ is generic (at a root of unity one should
use $\A$-forms in the argument below). Note that by $i)$ and $ii)$
of Proposition-Definition \ref{Ur prop def} we see that
multiplication defines an embedding $\Uq(\uobar) \ot \Uq(\lo)^\fin
\ot \Uq(\uo) \to \UqP$. We must show it is surjective.

Let $v \in \UqP$. We can decompose $v = \sum x_i \ot y_j \ot z_k$
according to the isomorphism of Proposition-Definition \ref{Ur
prop def} $ii)$ where we can assume that the $x_i$'s $y_j$'s and
$z_k$'s are linearly independent weight vectors. We must show that
each $y_j \in \Uq(\lo)^\fin$.

Assume in order to get a contradiction that there is a $j_0$ such
that $y_{j_0} \in \Uq(\lo) \setminus \Uq(\lo)^\fin$. Thus there is
WLOG an $E = E_\alpha$ such that $ad_E := ad_r(E)$ is not locally
finite on $y_{j_0}$; this implies that for all $s \geq 1$ we have
$$(*) \ \ \ \ \ ad_E^{s}(y_{j_0}) \notin
\operatorname{Span}\{ad_E^{t}(y_{j_0}); t < s\}.$$ We claim that
this implies that $ad_E$ is not locally finite on $v$. By
subtracting all summands $x_i \ot y_j \ot z_k$ for which $ad_E$ is
locally finite on $y_j$ we can assume that $ad_E$ is not locally
finite on any $y_j$ and hence we can assume that there is a vector
$x_{i_0} \ot y_{j_0} \ot z_{k_0}$ such that $x_{i_0}$ has lowest
weight among all the $x_i$'s. But then $ad^{s}_E(v)$ contains a
term $K^{-s}x_{i_0} \ot E^{s} y_{j_0} \ot z_{i_0}$, where $K =
K_\alpha$, which by $(*)$ isn't cancelled by the other terms. This
gives the desired contradiction.
\end{proof}

\begin{rems}\label{rems} A) It follows from the proof of Proposition-Definition \ref{Ur prop def} that the case $\n = \uo$ gives
$\Uq(\n) \overset{def}{=} \C \langle K_{-\alpha} E_\alpha; \alpha
\in \Delta \rangle$. This definition is not the standard one:
usually one takes $\Uq(\n)$ to be $\C \langle E_\alpha; \alpha \in
\Delta \rangle$. It follows however from the Serre relations that
our $\Uq(\n)$ is isomorphic to the latter algebra.

\medskip

\noindent B) Observe that $\Uq(\uo)_{>0}$ annihilates every finite
dimensional irreducible representation of $\Uq(\p)$. Moreover,
$\Uq(\bo) \cdot \Uq(\uo)_{>0}$ is generated as a left
$\Uq(\bo)$-ideal by $E_\alpha$ and $E_\alpha E_\beta$, for $\alpha
\in \Delta \setminus \Delta_P$ and $\beta \in \Delta_P$.

\medskip

\noindent C) The result of Corollary \ref{finite parabolic deco}
is optimal in the sense that it is impossible to construct a
$\p$-parabolic triangular decomposition of $\Uq^{\lo'\text{-fin}}$
for $\lo'$ a Levi such that $\lo \subsetneqq \lo'$.
\end{rems}

\end{document}